\documentclass[12 pt]{amsart}
\usepackage{amscd,amsmath,amsthm,amssymb}
\usepackage{verbatim}
\usepackage[all]{xy}
\usepackage{color}
\usepackage{hyperref}
\usepackage{tikz, float} \usetikzlibrary {positioning}

\usepackage[lined,algonl,boxed,norelsize]{algorithm2e}
\let\chapter\undefined

\def\II{{\NZQ I}}
\def\frk{\mathfrak}               % font for "Fraktur"

\def\Phi{{\frk n}}
\def\Phi{{\frk N}}

\def\II{{\mathbf I}}

\def\HS{{\mathcal HS}}

\def\kb{{\mathbf k}}

\def\A{{\mathcal A}}

\def\H{{\mathcal H}}

\def\opn#1#2{\def#1{\operatorname{#2}}} % to make operators
\opn\chara{char} \opn\length{\ell} \opn\pd{pd} \opn\rk{rk}
\opn\projdim{proj\,dim} \opn\injdim{inj\,dim} \opn\rank{rank}
\opn\depth{depth} \opn\grade{grade} \opn\height{height}
\opn\embdim{emb\,dim} \opn\codim{codim}

\opn\Tr{Tr} \opn\bigrank{big\,rank}
\opn\superheight{superheight}\opn\lcm{lcm}
\opn\trdeg{tr\,deg}%\emph{
\opn\reg{reg} \opn\lreg{lreg} \opn\ini{in} \opn\lpd{lpd}
\opn\size{size} \opn\sdepth{sdepth}
\opn\link{link}\opn\fdepth{fdepth}\opn\lex{lex}
\opn\LM{LM}
\opn\LC{LC}
\opn\NF{NF}
\opn\Merge{Merge}
\opn\sgn{sgn}
\opn\div{div} \opn\Div{Div} \opn\cl{cl} \opn\Pic{Pic}
\opn\Prin{Prin}
\opn\op{op}
\opn\indeg{indeg} \opn\outdeg{outdeg}
\opn\red{red}
\opn\Spec{Spec} \opn\Supp{Supp} \opn\supp{supp} \opn\Sing{Sing}
\opn\Ass{Ass} \opn\Min{Min}\opn\Mon{Mon}
\opn\Ann{Ann} \opn\Rad{Rad} \opn\Soc{Soc}
 \opn\Ker{Ker} \opn\Coker{Coker} \opn\Am{Am}
\opn\Hom{Hom} \opn\Tor{Tor} \opn\Ext{Ext} \opn\End{End}
\opn\Aut{Aut} \opn\id{id}

\opn\nat{nat}
\opn\pff{pf}%   \pf exists already
\opn\Pf{Pf} \opn\GL{GL} \opn\SL{SL} \opn\mod{mod} \opn\ord{ord}
\opn\Gin{Gin} \opn\Hilb{Hilb}\opn\sort{sort}
\opn\span{span}
\opn\Image{Image}
\opn\aff{aff} \opn\con{conv} \opn\relint{relint} \opn\st{st}
\opn\lk{lk} \opn\cn{cn} \opn\core{core} \opn\vol{vol}
\opn\link{link} \opn\star{star}\opn\lex{lex}\opn\set{set}
\opn\dist{dist}
\opn\gr{gr}

\def\pot#1#2{#1[\kern-0.28ex[#2]\kern-0.28ex]}

\opn\dirlim{\underrightarrow{\lim}}
\opn\inivlim{\underleftarrow{\lim}}
\def\prob{{\rm prob}}
\def\cdf{{\rm cdf}}
\def\Implies{\ifmmode\Longrightarrow \else
        \unskip${}\Longrightarrow{}$\ignorespaces\fi}
\def\implies{\ifmmode\Rightarrow \else
        \unskip${}\Rightarrow{}$\ignorespaces\fi}
\def\iff{\ifmmode\Longleftrightarrow \else
        \unskip${}\Longleftrightarrow{}$\ignorespaces\fi}

\let\:=\colon
\newtheorem{Theorem}{Theorem}[section]
\newtheorem{Lemma}[Theorem]{Lemma}

\newtheorem{Proposition}[Theorem]{Proposition}

\theoremstyle{remark}
\newtheorem{Remark}[Theorem]{Remark}

\theoremstyle{definition}
\newtheorem{Example}[Theorem]{Example}

\let\kappa=\varkappa
\def\qed{\ifhmode\textqed\fi
      \ifmmode\ifinner\quad\qedsymbol\else\dispqed\fi\fi}
\def\textqed{\unskip\nobreak\penalty50
       \hskip2em\hbox{}\nobreak\hfil\qedsymbol
       \parfillskip=0pt \finalhyphendemerits=0}
\def\dispqed{\rlap{\qquad\qedsymbol}}

\opn\dis{dis}
\def\pnt{{\raise0.5mm\hbox{\large\bf.}}}

\opn\Lex{Lex}
\opn\syz{{\rm syz}}
\opn\spoly{{\rm spoly}}
\opn\LM{{\rm LM}}
\opn\lm{{\rm lm}}
\opn\lcm{{\rm lcm}} \opn\A{\mathcal A}
\numberwithin{equation}{section}

\tikzstyle{Cwhite}=[scale = .6,circle, fill = white, minimum size=2.5mm] 
\tikzstyle{Cgray}=[scale = .4,circle, fill = gray, minimum size=3mm] 
\tikzstyle{Cblack2}=[scale = .4,circle, fill = black, minimum size=3mm] 
\tikzstyle{Cblack}=[scale = .7,circle, fill = black, minimum size=3mm]
\tikzstyle{C0}=[scale = .9,circle, fill = black!0, inner sep = 0pt, minimum size=3mm]
\tikzstyle{C1}=[scale = .7,circle, fill = black!0, inner sep = 0pt, minimum size=3mm]
\tikzstyle{Cred}=[scale = .4,circle, fill = red, minimum size=3mm] 
\tikzstyle{Cblue}=[scale = .4,circle, fill =blue, minimum size=3mm] 
\title{Types of signature analysis in reliability \\ based on Hilbert series}
\author{Fatemeh Mohammadi}
\address{$^{1}$Institut f\"ur Mathematik, Technische Universit\"at Berlin, 10623 Berlin, Germany }
\email{fatemeh.mohammadi@math.tu-berlin.de}

\author{Eduardo S\'aenz-de-Cabez\'on}
\address{Departamento of Matem\'aticas y Computaci\'on,  Universidad de La Rioja, Spain}
\email{eduardo.saenz-de-cabezon@unirioja.es}

\author{Henry P. Wynn}
\address{Department of Statistics, London School of Economics, UK}
\email{h.wynn@lse.ac.uk}
\date{}                                           % Activate to display a given date or no date

\begin{document}
\maketitle
%\tableofcontents
\begin{abstract}
The present paper studies multiple failure and signature analysis of coherent systems using the theory of monomial ideals. While system reliability has been studied using Hilbert series of monomial ideals, this is not enough to understand in a deeper sense the ideal structure features that reflect the behavior of the system under multiple simultaneous failures and signature. Therefore, we introduce the lcm-filtration of a monomial ideal, and we study the Hilbert series and resolutions of the corresponding ideals. Given a monomial ideal, we explicitly compute the resolutions for all ideals in the associated lcm-filtration, and we apply this to study coherent systems. Some computational results are shown in examples to demonstrate the usefulness of this approach and the computational issues that arise. We also study the failure distribution from a statistical point of view by means of the algebraic tools described.
\end{abstract}

\section{Algebraic reliability summary}
Let $m$ be a positive integer and consider a coherent system $S$ with $m$ components, each of which can be in a finite number of states. The set of possible states of  the whole system $S$ can be coded as elements of $\mathbb N^m$. The set of possible states contains a distinguished subset $\mathcal F$ of failure states which we assume to be coherent, meaning closed above under the standard entrywise ordering. The assumption of coherence is equivalent with saying that the failure states are precisely the exponents appearing in the monomials of the monomial ideal $\mathcal M_{\mathcal F}\subseteq \kb[x_1,\dots,x_m]$.

Fix a probability distribution on the set of states such that only finitely many of the states have positive probability. In this setting, the main object of interest from the point of view of reliability theory is the failure probability, defined as $\mathbb{E}[I_\mathcal{F}(\alpha)]$ where $I_\mathcal{F}(\alpha)$ is the indicator function of the failure set $\mathcal F$.

The indicator function and the multigraded Hilbert series of the ideal $M_\mathcal{F}$ are closely related. Specifically, if $\mathbb F$ is any free resolution of $M_{\mathcal F}$ with multigraded ranks $\gamma_{i,\mu}$ we have
\begin{eqnarray}\label{hilb}
{\HS_I(t,x)}=\frac{1+\sum_{i=1}^d (-1)^i t^i(\sum_{\gamma \in {
\mathbb{N}^n}} \gamma_{i,\alpha} x^{\alpha})}{\prod_{j=1}^n(1-x_i)}.
\end{eqnarray}
%where
%$\beta_{i,\alpha}$ depend only on $I$ and are known as the
%\emph{multigraded Betti numbers} of $S/I$. 
To simplify our notation, we set
\begin{eqnarray}
 \H_I(t,x)=-\sum_{i=1}^d (-1)^i t^i(\sum_{\beta \in {\mathbb{N}^n}} \gamma_{i,\alpha} x^{\alpha}), \
\end{eqnarray}
and we refer this as the numerator of the Hilbert series of $I$ which can be seen as a special kind of inclusion-exclusion formulae for counting the monomials in the union of the
ideals based on each individual minimal generator. 
We also set
\begin{eqnarray}\label{eq:hilb}
 \H_I(1,\alpha)=I_{\mathcal F}(\alpha)=-\sum_{i=1}^m(-1)^i\sum_{\mu\in\mathbb N^m}\gamma_{i,\mu}I_{x^\mu}(\alpha),
\end{eqnarray}
where $I_{x^\mu}(\alpha)$ is the indicator function of the states $\alpha\geq\mu$. If $\mathbb F$ is the minimal free resolution of $M_\mathcal{F}$ then the ranks $\gamma_{i,\mu}$ are the smallest possible and depend only on $M_\mathcal{F}$, in this case we call them the {\em multigraded Betti numbers} of $M_\mathcal{F}$ and denote them by $\beta_{i,\mu}$. The formula \eqref{eq:hilb} is potentially useful for determining failure probabilities because it reduces the problem to the computation of the simpler probabilities $\mathbb{E}[I_{x^\gamma}(\alpha)]$. By truncating the sum over $i$ we obtain obvious lower and upper bounds for $I_{\mathcal F}(\alpha)$ and thus on failure probabilities. These bounds improve the traditional Bonferroni bounds based on inclusion-exclusion. Among the bounds coming from free resolutions, the tightest ones are obtained when $\mathbb F$ is taken to be the minimal free resolution of $M_{\mathcal F}$.

From these basic principles, the authors have studied different aspects of the relation between monomial ideals and coherent systems in a series of works. In \cite{SW10,SW11} several relevant systems, including $k$-out-of-$n$ systems and variants, or series-parallel systems were studied, obtaining explicit and recursive formulas for the Betti numbers of the ideals of those systems. In \cite{SW15} the Hilbert function of the system ideal is applied to optimal design in reliability. Other works extending the scope of application of this approach are \cite{SW14} and \cite{FEH}  devoted to robustness measure of networks and percolation on trees respectively.

In this paper we propose two further steps for the application of multigraded Hilbert series and functions in probability, namely the study of multiple simultaneous failures of the system, and signature analysis. These two problems, not totally unrelated, imply a deeper knowledge of the structure of the system under consideration, beyond the information given by reliability analysis. The algebraic approach has already proven useful for the analysis of the structure of coherent systems, not only in what respects the reliability of the system but also on system design and measures of components importance \cite{SW10,SW15}. But the tools used so far, namely the Hilbert series of the ideal of the system are not enough when one needs to study simultaneous failures or signature analysis. We therefore introduce a new algebraic object that provides the necessary insight on the ideals structure: the $\lcm$-filtration. With this motivation at hand, the main contribution of the paper is the definition of the $\lcm$ filtration of a monomial ideal, the study of its resolutions and Hilbert series and also the performance of actual computations on these objects. The results obtained are then successfully applied to the study of simultaneous failures and signatures in two different paradigmatic examples.

The plan of the paper is the following. In \S\ref{sec:twoSteps} we present the two problems that we are dealing with: multiple failure and signature analysis. In this section, we see that the main algebraic object we need to understand to study these problems is the $\lcm$-filtration of a monomial ideal. The $\lcm$-filtration is studied in \S\ref{sec:lcmFiltration} where we give an explicit free resolution for the ideals involved in the $\lcm$-filtration of any monomial ideal $I$. This section uses previous work by Aramova, Herzog and Hibi \cite{AHH97, AHH98} on square free stable ideals, and by Peeva and Velasco \cite{PV05} on frames of monomial resolutions. We study also the behavior of this and other resolutions in examples. Finally, \S\ref{sec:statistics} is devoted to estimation of the probability distribution of the multiple failures of a system. During the paper we use two different, paradigmatic examples, consecutive $k$-out-of-$n$ systems and cut ideals of complete graphs. The structural differences of the behavior of these two systems with respect to multiple failures is made evident by the use of the concepts and techniques developed in the paper.

\section{Two steps further}\label{sec:twoSteps}

\subsection{Multiple failures}\label{sec:multipleFailures}
%OLD NAME:{Filtration and the survivor}

Let $S$ be a coherent system in which several minimal failures can occur at the same time. Let $Y$ be the number of such simultaneous failures. The event $\{Y\geq 1\}$ is the event that at least one elementary failure event occurs, which is the same as the event that the system fails. If $x^\alpha$ and $x^\beta$ are the monomials corresponding to two elementary failure events then $\lcm(x^\alpha,x^\beta)=x^{\alpha\wedge\beta}$ corresponds to the intersection of the two events and we have $Y\geq 2$. The corresponding ideal is $\langle x^\alpha\rangle\cap \langle x^\beta\rangle$. The full event $Y\geq 2$ corresponds to the ideal generated by all such pairs. The argument extends to $Y\geq k$  and to study the tail probabilities $\prob\{Y\geq k\}$. We now discuss these ideals in more detail.

\subsubsection{The $\lcm$-filtration and the survivor}
Let $I\subseteq \kb [x_1,\dots,x_n]$ be a monomial ideal and $\{m_1,\dots,m_r\}$ be a minimal monomial generating system of $I$. Let $I_k$ be the ideal generated by the least common multiples of all sets of $k$ distinct monomial generators of $I$,
\[
I_k=\langle \lcm(\{m_i\}_{i\in\sigma}) : \sigma\subseteq\{1,\dots,r\}, \vert\sigma\vert=k\rangle.
\]
We call $I_k$ the \emph {$k$-fold $\lcm$-ideal of $I$}. The ideals $I_k$ form a descending filtration 
\[
I=I_1\supseteq I_2\supseteq \cdots \supseteq I_r,
\]
which we call the \emph{$\lcm$-filtration} of $I$.

\medskip

The survivor functions $$F(k)=\prob\{Y\geq k\}$$ for a coherent system, are obtained from the multigraded Hilbert function of the $k$-fold $\lcm$-ideal $I=I_k$. In fact, to emphasize the counting we relabel $I_{\mathcal F}$ as $I_1$.

\begin{Example}
Consider a sequential (also named consecutive) $k$-out-of-$n$ system with $n=5, k=2$. Then 
\[
I_1=\langle x_1x_2,x_2x_3,x_3x_4,x_4x_5\rangle.
\]
The numerator of the Hilbert series obtained from the Taylor resolution of $I_1$ is formed by successively taking the $\lcm$'s of pairs, triples and so on, with sign changes and with cancellations across neighboring rows,  $I_F(\alpha)$ is the evaluation of the polynomial
\[
x_1x_2+x_2x_3+x_3x_4+x_4x_5\\
-(x_1x_2x_3+x_2x_3x_4+x_3x_4x_5+x_1x_2x_4x_5)+x_1x_2x_3x_4x_5
\]
which is the numerator of the Hilbert series of $I_1$. The Taylor resolution (which is equivalent to full inclusion-exclusion formula) uses terms from the full $\lcm$-lattice. A similar analysis gives the numerator of the Hilbert series of $I_2$ as:
\[
x_1x_2x_3+x_2x_3x_4+x_3x_4x_5+x_1x_2x_4x_5-(x_1x_2x_3x_4+x_2x_3x_4x_5+x_1x_2x_3x_4x_5).
\]
\end{Example}

The above considerations can be summarized in the following lemma.
\begin{Lemma}
Let $Y$ be the number of failure events. If $G(M_{\mathcal F})$ is the set of minimal generators of the failure ideal $I_{\mathcal F}=I_1$,
then 
\[
\mathbb{P}\{Y\geq k\}=\mathbb{E}[I_{M_k}(\alpha)]=\H_{I_k}(1,\alpha),
\]
where $I_{M_k}$ is the indicator function of the exponents of monomials in the $k$-fold $\lcm$-ideal
\[
I_k:=\langle \lcm(\{m_i\}_{i\in\sigma}):\sigma\subseteq G(M_{\mathcal F}),\vert\sigma\vert=k\rangle.
\]
\end{Lemma}
As a result, we also obtain identitites, lower and upper bounds for multiple failure probabilities from free resolutions in \S\ref{sec:lcmFiltration}.

\subsection{Signature analysis}\label{sec:signatureAnalysis}
\subsubsection{Classical Signature}
The theory of signature analysis considers a system with $m$ components which fail independently with a common failure time distribution with cumulative distribution function, $\cdf$, $\mathcal{F}(t)$ and density $f(t)$. As time proceeds, components start to fail and one can write the failure times as
\[
T_{(1)},T_{(2)},\dots,T_{(m)}.
\]
In statistical terminology, these are the order statistics of the full set of failure times. Because of the distributional assumption there are no ties. Now at some (first) integer $i$ the system will fail: $T=T_{(i)}$ where $T$ is the failure time of the system. The signature $s_i$ codes the probabilities
\begin{eqnarray*}
s_i&=&\prob\{T=T_{(i)}\}\\
&=&\prob\{T\geq T_{(i)}\}-\prob\{T\geq T_{(i+1)}\},\quad i=1,\dots,m.
\end{eqnarray*}
Moreover, because the system must eventually fail, $\sum_{i=1}^n s_i =1$. For material on signature analysis see \cite{Boland,Samaniego}. A main idea of this paper is that we can derive $s_1,\ldots,s_m$ from the failure ideal. The value $s_i$ is the conditional probability that exactly $i$ components have failed, \emph{conditional} on the event that the system has failed. If we use a squarefree (binary) representation, then for the monomials describing individual failures the degree gives the number of component failures. Thus, if $P_i=\sum_{\alpha\in \mathcal{F},\vert\alpha\vert=i}p_\alpha$ and $P(\mathcal{F})=\sum_{\alpha\in \mathcal{F}}p_\alpha$ (i.e., the full probability of failure), then
\[
s_i=\frac{P_i}{P(\mathcal{F})}.
\]
Identitites,  upper and lower bounds for the $P_i$ are inherited from those for the $P(\mathcal{F})$ simply by intersecting with the event $A_i=\{\alpha : \vert\alpha\vert=i\}$, which can be found by extracting all the terms of the same degree $i$.
\medskip

We have that $s_i=\frac{\mathbb{E}[I_{E_i}(\alpha)]}{\mathbb{E}[I_{\mathcal F}(\alpha)]}$, where $E_i$ is the set consisting of $\alpha$'s for which exactly $i$ elements have failed and $\alpha$ is a failure state.
\medskip

\noindent
{\bf Notation.}
Let $I=I_S\subset \kb[x_1,\dots,x_m]$ be the failure ideal of the system $S$. We denote $I^{[i]}$ for the set of squarefree monomials of degree $i$ in $I$, and $\langle I^{[i]}\rangle$ for the ideal generated by all such  monomials. The indicator function $I_{E_i}$ is the indicator function of the set of exponents of the monomials in $I^{[i]}$.

\medskip

Observe that now the exact formulas for the probabilities $s_i$ as well as upper and lower bounds can be obtained from the free resolutions and Hilbert series of $\langle I^{[i]}\rangle$. In particular, $s_i$ is  the difference of the evaluation of the numerator of the Hilbert series of  $\langle I^{[i]}\rangle$ and  $\langle I^{[i+1]}\rangle$.

\medskip
\begin{comment}
The main theorem is that the signature theory is manifested algebraically via the \emph{intersection of the system failure ideal with the ideal for a $k$-out-of-$n$ system with $k=i$, $n=m$}. The corollaries give bounds derived for failure probability, including those from the $\lcm$ theory. These comparison of these different partitions of the failure, by the number of minimal failure events and the number of component failures, yield considerable information about the robustness of the system.
\end{comment}

\subsubsection{The $k$-fold signature}
Now we focus on multiple simultaneous minimal failures on the system. 
As time proceeds, components start to fail simultaneously and one can write the failure times as
\[
T_{(1)}^{k},T_{(2)}^k,\dots,T_{(m)}^k,
\]
where $T_{(i)}^k$ is the time in which we have $k$ simultaneous minimal failures with $i$ failed components.
As before, at some (first) integer $i$ the system will have $k$ simultaneous minimal failures: $T^k=T_{(i)}^k$ where $T^k$ is the  time of $k$ minimal failures in the system. The $k$-signature $s_i^k$ codes the probabilities
\begin{eqnarray*}
s_i^k&=&\prob\{T^k=T_{(i)}^k\}\\
&=&\prob\{T^k\geq T_{(i)}^k\}-\prob\{T^k\geq T_{(i+1)}^k\},\quad i=1,\dots,m,\ {\rm and}\ k=1,\ldots,r.
\end{eqnarray*}

Now  the ideal encoding the states that at least $k$ failures occur simultaneously is the $k$-fold $\lcm$-ideal $I_k$. Doing the same analysis as we did for the classical signature, we obtain that $s_i^k$ is  the difference of the evaluation of the numerator of the Hilbert series of  $\langle I_k^{[i]}\rangle$ and  $\langle I_k^{[i+1]}\rangle$. This is why we call $s_i^k$ the $k$-fold signature of the system.

%\subsubsection{\textcolor{blue}{Multivariate signature}}
%\subsubsection{\textcolor{blue}{The joint distribution}}

\section{The $\lcm$-filtration and its resolutions}\label{sec:lcmFiltration}

\subsection{Aramova-Herzog-Hibi resolution and frames}

%\subsubsection{Generalized Taylor resolutions }

%$I_{k,n}\subseteq k[x_1,\dots,x_n]$ is the ideal of a $k$-out-of-$n$ system, generated by all products of $k$ of the $n$ variables.
Let $I\subseteq k[x_1,\dots,x_d]$ be a monomial ideal with $r$ generators, and let 
$I_k$ be the $k$-fold $\lcm$-ideal of $I$. % generated by all $k$-fold $\lcm$'s of generators of $I$.
In this section we want to study various free resolutions of $I_k$. We will use an explicit minimal resolution for the ideal generated by all $k$-fold products of $r$ variables, which is called the $k$-out-of-$r$ ideal. Using frame theory, as developed by Peeva and Velasco in \cite{PV05}, we  construct resolutions of $I_k$, and relate them to the Taylor resolution \cite{T66}.

\subsubsection{The minimal free resolution of the $k$-out-of-$r$ ideal.}
Let $I_{k,r}\subseteq R=\kb[x_1,\dots,x_r]$ be the ideal generated by all products of $k$ different variables. This ideal corresponds in algebraic reliability theory, to the ideal of a $k$-out-of-$r$ system \cite{SW11}. It is clear that $I_{k,r}$ is a squarefree stable ideal in the sense of \cite{AHH98}, hence its minimal free resolution is of the same form of the Eliahou-Kervaire resolution \cite{EH90} as described explicitly by Aramova, Herzog and Hibi in \cite{AHH98}. Let us recall here some definitions and notation from \cite{AHH98}. For every squarefree monomial ideal $I$ minimally generated by $G(I)$ and for any squarefree monomial in $I$, there exists a unique pair $(a,b)$ of squarefree monomials in $R$ such that $a\in G(I)$, $m=a\cdot b$ and $\max(a)<\min(b)$. Thus, we have a map $g$ from the set of squarefree monomials in $I$ to the set $G(I)$. This map is given by $g(m)=a$. Now, given $j\in \{1,2,\dots,r\}$ with $j\in\supp(m)$ we set $m_j=g(x_j\cdot m)$ and $y(m)_j=(x_j\cdot m)/m_j$. Aramova, Herzog and Hibi gave an explicit resolution of any squarefree monomial ideal using the function $g$. The cellular realizations of these resolutions are described in \cite{AntonFatemeh}.

\medskip

We give here an explicit description of the Aramova-Herzog-Hibi resolution of $I_{k,r}$ independent of the map $g$ and
based only on the subsets of $\{1,\dots,r\}$.
\begin{Proposition}
Let $I_{k,r}\subseteq \kb[x_1,\dots,x_r]$ be the $k$-out-of-$r$ ideal. The minimal free resolution of $I_{k,r}$ is given by
\[
\II_{k,r}: 0\longrightarrow F_{r-k}\stackrel{\partial}{\longrightarrow}\cdots\stackrel{\partial}{\longrightarrow}F_1\stackrel{\partial}{\longrightarrow}I_{k,r}\longrightarrow 0,
\]
where
\begin{enumerate}
\item{Each generator of $F_i$ is labelled by a pair $[\sigma,\tau]$ such that $\sigma,\tau\subseteq\{1,\dots,r\}$, $\sigma\cap\tau=\emptyset$, $\vert\sigma\vert=k$, $\vert\tau\vert=i$ and $\max(\sigma)>\max(\tau)$.}
\item{The differential $\partial$ is given by 
\[
\partial([\sigma,\tau])=\sum_{j\in\tau}(-1)^{\sgn(j,\tau)}(-x_j[\sigma,\tau-j]+x_{\max(\sigma)}[\sigma- \max(\sigma)+j,\tau-j])
\]

if $\max(\tau)<\max(\sigma-\max(\sigma))$, or

\[
\partial([\sigma,\tau])=\sum_{j\in\tau}(-1)^{\sgn(j,\tau)}(-x_j[\sigma,\tau-j])+x_{\max(\sigma)}[\sigma- \max(\sigma)+\max[\tau],\tau-\max(\tau)]
\]
if $\max(\tau)>\max(\sigma-\max(\sigma))$. Observe that in the first case we have $2\vert\tau\vert$ summands, and in the second case only $\vert\tau\vert+1$ summands.}

\end{enumerate}
\end{Proposition}

\begin{proof}
Since $I_{k,r}$ is squarefree stable, we can use the description of its minimal free resolution given in \cite{AHH98} which has an equivalent labeling of the generators of each free module $F_i$. Using our notation, this differential is
\[
\delta([\sigma,\tau])=\sum_{j\in\tau}(-1)^{\sgn(j,\tau)} (-x_j[\sigma,\tau-j]+y(\sigma)_j[\sigma_j,\tau-j]),
\]
where $\sigma_j=g(\sigma+j)$, $y(\sigma)_j=(\sigma+j)-\sigma_j$ and $[\mu,\rho]=0$ if $\max(\rho)>max(\mu)$.

\noindent
To complete the proof we have to show that
\begin{enumerate}
\item For all $j\in\tau$, $y(\sigma)_j=\max(\sigma)$ and $\sigma_j=\sigma-\max(\sigma)+j$
\item If $\max(\tau)>\max(\sigma-\max(\sigma))$ and $j<\max(\tau)$, then $\max(\tau-j)>\max(\sigma-\max(\sigma)+j)$ and hence $[\sigma-\max(\sigma)+j]=0$.
\end{enumerate}

To prove (1) observe that for $\sigma+j$ the unique $\delta$, $\gamma$ such that $\sigma+j=\delta\cup\gamma$ with $\delta \in G(I_{s,r})$ and $\max(\delta)<\min(\gamma)$ are $\delta=\sigma-\max(\sigma)+j$ and $\gamma=\max(\sigma)$: First, it is clear that $\delta\cup\gamma=\sigma+j$, and since $j\leq\max(\tau)<\max(\sigma)$ then $\max(\delta)<\min(\gamma)$. Then, by definition $\sigma_j=\delta=\sigma-\max(\sigma)+j$ and $y(\sigma)_j=\sigma+j-\sigma_j=\sigma+j-\sigma+\max(\sigma)-j=\max(\sigma)$.

To prove (2) first we note that if $j=\max(\tau)$, then $\max(\sigma-\max(\sigma)+j)\geq j$ and $\max(\tau-j)<j$. Thus $\max(\sigma-\max(\sigma)+j)>\max(\tau-j)$. Now if $j<\max(\tau)$, then on one hand we have that $\max(\sigma-\max(\sigma))<\max(\tau)$ implies that $\max(\sigma-\max(\sigma)+j)$ is equal to $\max(\max(\sigma-\max(\sigma)),j)$, which is strictly less than $\max(\tau)=\max(\tau-j)$. Hence we have that $[\sigma-\max(\sigma)+j,\tau-j]=0$. On the other hand, if $\max(\sigma-\max(\sigma))>\max(\tau)$ then $\max(\sigma-\max(\sigma)+j)=\max(\sigma-\max(\sigma))$ is strictly greater than $\max(\tau)=\max(\tau-j)$.
This completes the proof. 
\end{proof}

Using $\II_{k,r}$ we construct now a resolution of $I_k$. We use the techniques and terminology of \cite{PV05}. First recall that the $\lcm$ lattice of a monomial ideal $M$ denoted by $L_M$ is the lattice whose elements labeled by the least common multiples of the monomial minimal generators of $M$. A monomial ideal $M$ in a polynomial ring $S$ is called a {\em reduction} of another monomial ideal $M'$ in a polynomial ring $S'$ over the same ground field of $S$, if there exists a map $f: L_{M'}\rightarrow L_M$ which is a bijection on the atoms and preserves $\lcm$'s. Such a map is called a \emph{degeneration}. 

\begin{Lemma}
$I_k$ is a reduction of $I_{k,r}$.
\end{Lemma}

\begin{proof}
Consider the map $f: L_{I_{k,r}}\rightarrow L_{I_k}$ that takes $\sigma$ to $m_{\sigma}$, where $m_{\sigma}=\lcm(\{m_i\}_{i\in\sigma})$. The map $f$ is clearly a degeneration and hence $I_k$ is a reduction of $I_{k,r}$. Observe that $f$ is not an isomorphism in general.
\end{proof}

We can now construct the $f$-degeneration $f(\II_{k,r})$ following \cite[Construction 4.3]{PV05}, which by  \cite[Theorem 4.6]{PV05} is a free resolution of $I_k$. We can also construct the $f$-homogenization of $\II_{k,r}$ by  \cite[Construction 4.10]{PV05} and since $\II_{k,r}$ is the minimal free resolution of $I_{k,r}$, both $f$-degeneration and $f$-homogenization coincide. We denote the so obtained (non necessarily minimal) resolution of $I_k$ by $\II_k$. Let us describe it:

The elements of the basis of $\II_{k}$ are also labelled by pairs $[\sigma,\tau]$ such that $\sigma,\tau\subseteq\{1,\dots,r\}$, $\sigma\cap\tau=\emptyset$, $\vert\sigma\vert=s$, $\vert\tau\vert=i$ and $\max(\sigma)>\max(\tau)$. Observe that the multidegree of element $[\sigma,\tau]$ in $\II_k$ is $f(\sigma\cup\tau)$.
The differential $\partial$ is given by 
\[
\partial([\sigma,\tau])=\sum_{j\in\tau}(-1)^{\sgn(j,\tau)}(-\frac{m_{\sigma\cup\tau}}{m_{\sigma\cup\tau - j}}[\sigma,\tau-j]+\frac{m_{\sigma\cup\tau}}{m_{\sigma\cup\tau - \max(\sigma)}}[\sigma- \max(\sigma)+j,\tau-j])
\]
if $\max(\tau)<\max(\sigma-\max(\sigma))$, or

\[
\partial([\sigma,\tau])=\sum_{j\in\tau}(-1)^{\sgn(j,\tau)}-\frac{m_{\sigma\cup\tau}}{m_{\sigma\cup\tau - j}}[\sigma,\tau-j]+\frac{m_{\sigma\cup\tau}}{m_{\sigma\cup\tau - \max(\sigma)}}[\sigma- \max(\sigma)+\max[\tau],\tau-\max(\tau)]
\]
if $\max(\tau)>\max(\sigma-\max(\sigma))$. 

\begin{Theorem}
Let $I=\langle m_1,\dots,m_r \rangle$ be a monomial ideal.  $\II_k$ is a free resolution of $I_k$ for all $k$. The Betti numbers of $I_{k,r}$ are an upper bound for the Betti numbers of $I_k$.
\end{Theorem}

\begin{Remark}
$\II_1$ is the Taylor resolution of $I$.
\end{Remark}
%\noindent{\bf Proof:} $\square$

\subsubsection{Resolutions for $I_k$}
For any monomial ideal we can construct different free resolutions. A distinguished one is the unique (up to isomorphism) minimal free resolution, which is not alway easy to obtain computationally. For this and other reasons, one usually uses other non minimal resolutions, that are easier to obtain. A prominent example is Taylor resolution \cite{T66}. Taylor resolution begins with any set of monomial generators of the ideal (not necessary the minimal generating set) and the resolution is described combinatorially.
In the case of the ideals $I_k$ involved in the $\lcm$-filtration of $I$, there are two natural choices for Taylor resolutions, the one that uses the minimal generating set of $I_k$ and the one that uses the (redundant) generating set given by all $k$-fold $\lcm$'s of generators of $I$. We denote the first one by $\mathbb{T}_k$ and the second one by $\mathbb{T}^k$.  $\mathbb{T}_k$ is the usual Taylor resolution of $I_k$. We also have the minimal free resolution $\mathbb{M}_k$ and the above described resolution $\II_k$.

All three resolutions $\mathbb{T}_k$, $\mathbb{M}_k$ and $\II_k$ are subresolutions of $\mathbb{T}^k$. $\mathbb{M}_k$ is a sub resolution of $\mathbb{T}^k$, $\mathbb{T}_k$ and $\II_k$. But $\mathbb{T}_k$, the usual Taylor resolution of $I_k$, and $\II_k$ are not sub resolutions of each other. Therefore $\II_k$ is one of the rare examples of interesting non sub-Taylor resolutions of a monomial ideal, which also has non minimal first syzygies. This is related to Mermin's questions in \cite{M12}.

\subsection{Examples}
Let us describe now two examples that demonstrate the usefulness of the $\lcm$-filtration to detect structural differences in the ideals (i.e. the systems) under study. The first example shows ideals of consecutive linear $k$-out-of-$n$ systems, i.e. systems that fail whenever $k$ {\em consecutive} components out if its $n$ components fail. The second example is the cut ideal of the $n$-complete graph. This ideal models the behavior of all-terminal reliability of a system of $n$ components, i.e. the system fails whenever there are two disconnected nodes, that is when the graph is disconnected. The behavior of these two systems with respect to multiple simultaneous failures is well illustrated by the study of their $\lcm$-filtrations. These will be our running examples in this paper. We now introduce them and show the behavior of the resolutions of the ideals in their $\lcm$-filtrations.

\begin{Example}
Consecutive $2$-out-of-$n$ ideals:

The ideal of the consecutive linear $2$-out-of-$n$ system is minimally generated by products of consecutive pairs of the variables, $x_1x_2$ up to $x_{n-1}x_n$. This corresponds to the edge ideal of the line graph, which has been intensively studied (also from the algebraic reliability point of view, see \cite{SW11}) and for which the minimal free resolution and Betti numbers are known \cite{HV10}. Figure \ref{fig:consecutive} shows the behavior of $\mathbb{T}_k$ (green, circles), $\II_k$ (red, squares) and $\mathbb{M}_k$ (blue, diamonds) for the $\lcm$-filtrations of the ideals of consecutive $2$-out-of-$n$ systems for $n=10,11,12$. Each line corresponds to the full filtration of one ideal, the abscissa corresponds to the level $k$ of the filtration one the ordinate gives the logarithm of the {\em total size} of the resolution, understood as the sum of all the ranks of the modules in the resolution. In this example we can see that while the generating set of each $\mathbb{T}_k$ is smaller than the corresponding one of $\II_k$, the latter resolution is much closer to the minimal free resolution, except for the latest steps of the filtration.

\begin{figure}[htbp]
\begin{center}
\includegraphics[scale=0.5]{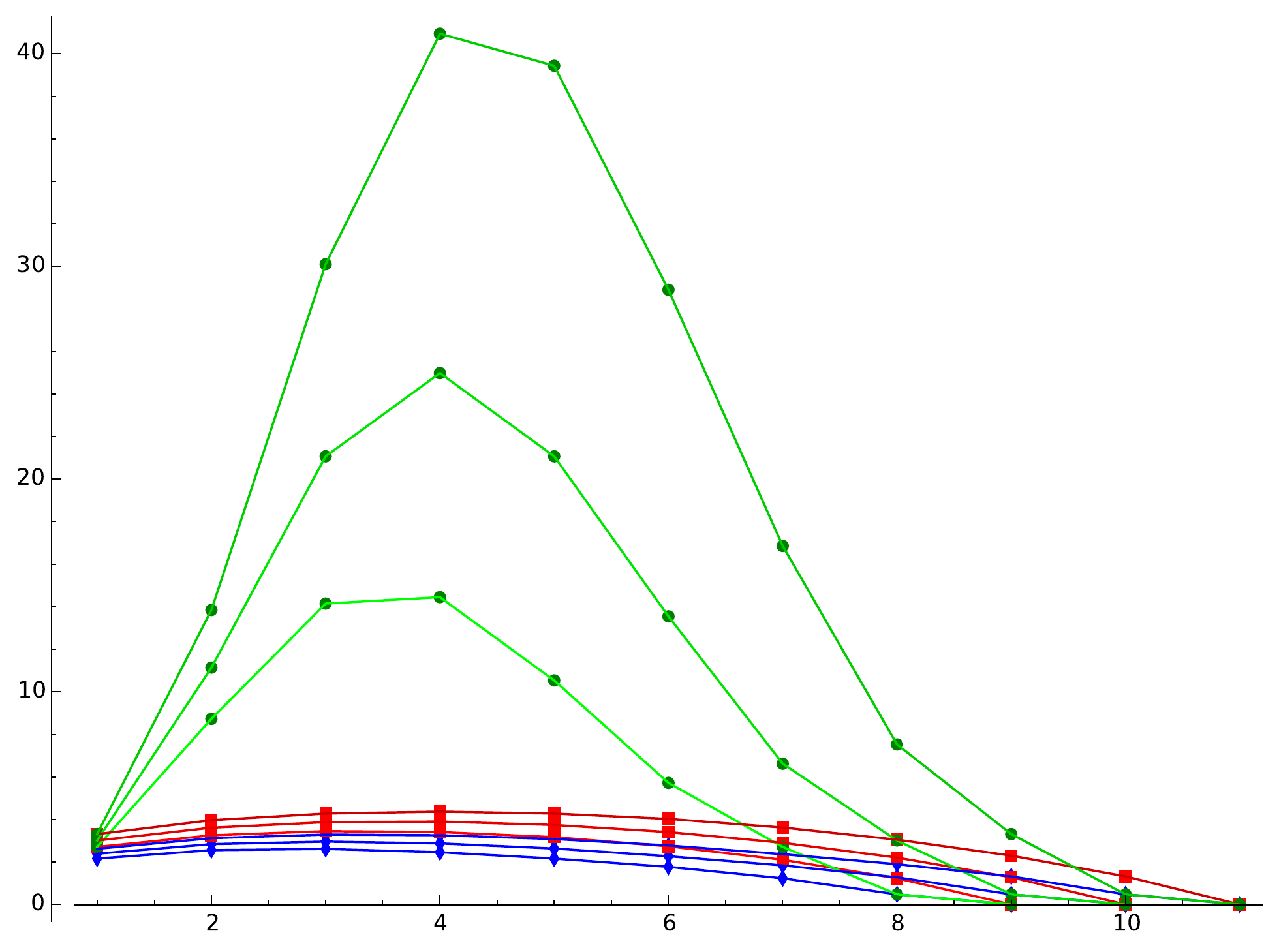}
\caption{Sizes of resolutions of $I_{2,n}$ for $n=10,11,12$}\label{fig:consecutive}
\end{center}
\end{figure}
\end{Example}

\begin{Example}% Cut ideals of graphs: 
Cut ideals of graphs: 

Figure \ref{fig:complete} shows the picture of the logarithm of the size of $\mathbb{T}_k$ (green, circles), $\II_k$ (red, squares) and $\mathbb{M}_k$ (blue, diamonds) for the ideals in the $\lcm$-filtration of the ideal of the complete graphs on $4$ and $5$ vertices. Observe that the picture is completely different from Figure \ref{fig:consecutive}. There are two main differences: In the first place, Taylor resolution is closer to the minimal one than $\II_k$ and that is because the cancellations of non minimal generators are much more numerous here than in the case of consecutive $k$-out-of-$n$ ideals, and therefore the difference of the resolutions is much more evident. Another difference is that the green and blue lines have some horizontal trends. This is because for some values of $k$, the $k$-fold lcm ideals are exactly the same, i.e. in this case the filtration has a staircase behavior, it remains constant for some steps and then shrinks and so on. This was not evident just from the generating set of the ideal and has to do with the nature of the cut ideals of the complete graph, which have been extensively studied in \cite{Fatemeh,FEH,FEH2,FatemehFarbod}. Here we just mention the main results that illustrate the behavior we have just observed, and we refer the interested readers to \cite{FEH2} for more details.

Let $K_n$ be the complete graph on $n$ vertices. Let $I_n\subset \kb[x_{ij}:\ 1\leq i<j\leq n]$ be the cut ideal of $K_n$. 
For integer $k$, $1\leq k\leq n$, we let $I_{n,k}$ be the $k$-fold $\lcm$-ideal of $I_n$. 
We denote by $\mathcal{P}_{n,k}$ the set of $k$-partitions of $[n]$. Note that the cardinality of $\mathcal{P}_{n,k}$ is the Stirling number of the second kind (i.e. the number of ways to partition a set of $n$ elements into $k$ nonempty subsets). For any $k$-partition of $[n]$, we associate a monomial whose support is the set of edges between distinct blocks of the partition. 
%We denote by $P_{n,k}$ the ideal minimally generated by the such monomials.
For example for the partition $\sigma=12|3|4$ of $K_4$ we associate the monomial  $m_\sigma=x_{13}x_{14}x_{23}x_{24}x_{34}$.
%Now  let $I_{n,k}$ be the $k$-fold $\lcm$-ideal of $I_n$, the cut ideal of $K_n$, the complete graph with the vertex set 
%$[n]$. The ideal $I_{n,k}$ is generated by the least common multiples of all sets of $k$ generators
%of $I_n$. 

We denote  by $P_{n,k}$ the ideal minimally generated by the monomials associated to the partitions in $\mathcal{P}_{n,k}$. 
We have the following relations between the ideals $I_{n,k}$ and the ideals $P_{n,k}$.

\begin{Theorem}\cite[Theorem~3.1]{FEH2}\label{thm:complete}
For all integeres $k$ and $n$, $1\leq k\leq n$ we have $$I_{2^{k-1}}=I_{2^{k-1}+1}=\cdots=I_{2^{k}-1}=P_{n,k+1}.$$
\end{Theorem}

\begin{figure}[htbp]
\begin{center}
\includegraphics[scale=0.5]{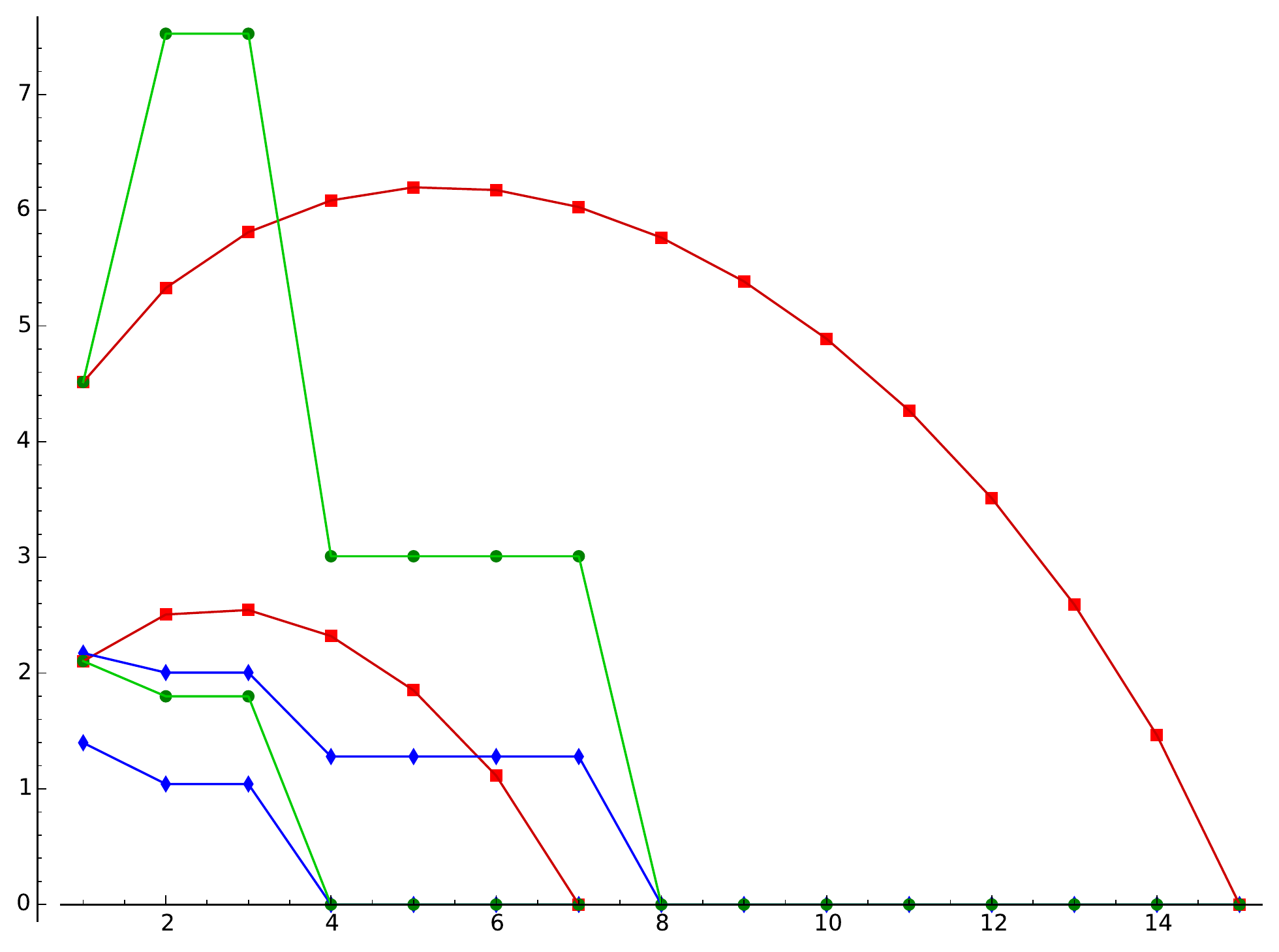}
\caption{Sizes of resolutions of the ideals in the $\lcm$ filtration of $I_{4}$ and $I_{5}$}\label{fig:complete}
\end{center}
\end{figure}

\end{Example}
%\begin{Example}
%Cut ideal of the complete graph $K_n$.
%\end{Example}

\subsection{Computational cost}
When dealing with the kind of computations we are presenting in this paper, one has to consider computational costs, in particular for applications. The problem of computing the list of multiple failure events of a system (equivalently, the $\lcm$s of the generators of a monomial ideal) is intrinsically expensive in term of memory (and time). This is mainly due to the fact that the cardinality of the set of $k$-fold $\lcm$s of the elements of a set grows exponentially with respect to the size of the set. This said, we can analyze the computations needed to compute the Betti numbers of the ideals in the $\lcm$-filtration in the ways proposed in this section. There are some interesting considerations in this respect.

For our analysis we will use our running examples, namely consecutive $2$-out-of-$n$ systems and the cut ideal of complete graphs. We will also use for our analysis a simple program written in the computer algebra system {\tt Macaulay2} \cite{M2} which computes the sizes of the Taylor, Aramova-Herzog-Hibi and minimal resolutions of the ideals in the $\lcm$-filtration of a given monomial ideal. The Taylor resolutions considered will use the minimal generating sets of the corresponding ideals. The main tasks needed are the following:
\begin{enumerate}
\item Computing the $k$-fold $\lcm$s of the generators of the ideal: The computation of the $k$-fold $\lcm$s of a set of monomials is a theoretically easy task, but since the number of such $\lcm$s is large, it is not computationally irrelevant. For this, we use the inbuilt command {\tt lcm} in {\tt Macaulay2}.
\item Obtaining the minimal set of generators of the ideal $I_k$: This step implies autoreduction of the set of generators obtained in the previous step, so that we obtain the minimal generating set. We use the command {\tt monomialIdeal} in {\tt Macaulay2} which, for a given set of monomials, builds the minimal generating set.
\item Computing the minimal free resolution of $I_k$: This is in principle the computationally most expensive task. For this we use the {\tt Macaulay2} command {\tt res} and assume we already have the minimal generating set of the corresponding ideal.
\end{enumerate}

Figures \ref{fig:computations1}  and \ref{fig:computations2} show the distribution of time of each of these three tasks in the $2$-out-of-$16$ and the cut ideal of complete graph $K_5$. We have chosen these examples because the number of generators of the original ideal is in both cases $15$. The $2$-out-of-$16$ ideal is generated in degree $2$, and the number of minimal generators of each of the corresponding ideals is relatively big.  On the other hand, the generators of the cut ideal of the complete graph have bigger degrees. In this case, there are much more cancellations between the generators of the $\lcm$-ideals, which are minimally generated by much fewer monomials. These differences have their reflection in the distribution of times of the three tasks mentioned above. The cut ideal of the complete graph uses most of its time in computing the $k$-fold $\lcm$-ideals. A good optimization of this part of the algorithm is therefore crucial for the application of this method to this kind of ideals. In the case of the $2$-out-of-$16$ ideal, on the contrary, the best part of the computing time is devoted to the computation of the minimal free resolutions. Observe that the times for both examples are quite similar, the differences are due to the different distribution of computing time among tasks, which reflects the differences in the $\lcm$-filtration of both examples.

The computation of the $k$-fold $\lcm$-ideals is unavoidable for our purposes, hence the data in our experiments show that for examples like the $2$-out-of-$16$ ideals, computing the Aramova-Herzog-Hibi resolutions makes the computation possible, for the time needed for the minimal free resolution is too big too soon. On the other hand, for examples like the cut ideal of the complete graph $K_5$, the computation time of the minimal free resolution is (at this size of examples) not so relevant, and one should focus in an efficient implementation of the construction of the $k$-fold $\lcm$-ideals. In this respect, Theorem \ref{thm:complete} is a strong result that on one hand reduces the number of ideals to be computed, and on the other hand allows us to substitute the computations of the subsequents ideals by computations of partitions of the set of vertices, for which we can adapt already existing efficient algorithms.The efficient computation of cut ideals for the complete graph is an important point towards the algebraic study of the  Erd\"os-R\'enyi model of random networks. For more details on this see \cite{FEH2}.

\begin{figure}[htbp]
\begin{center}
\includegraphics[scale=0.47]{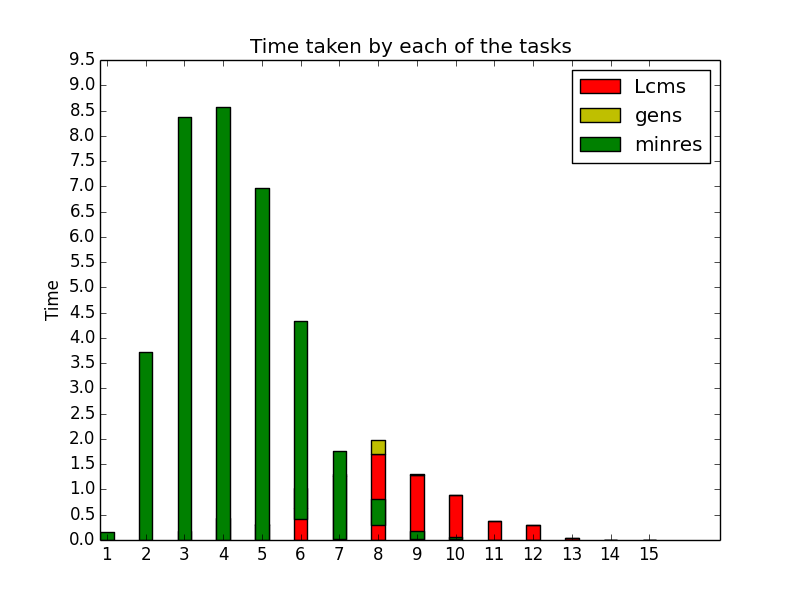}
\caption{Costs of each of the tasks for the consecutive $2$-out-of-$16$ ideal}
\label{fig:computations1}
\end{center}
\end{figure}

\begin{figure}[htbp]
\begin{center}
\includegraphics[scale=0.47]{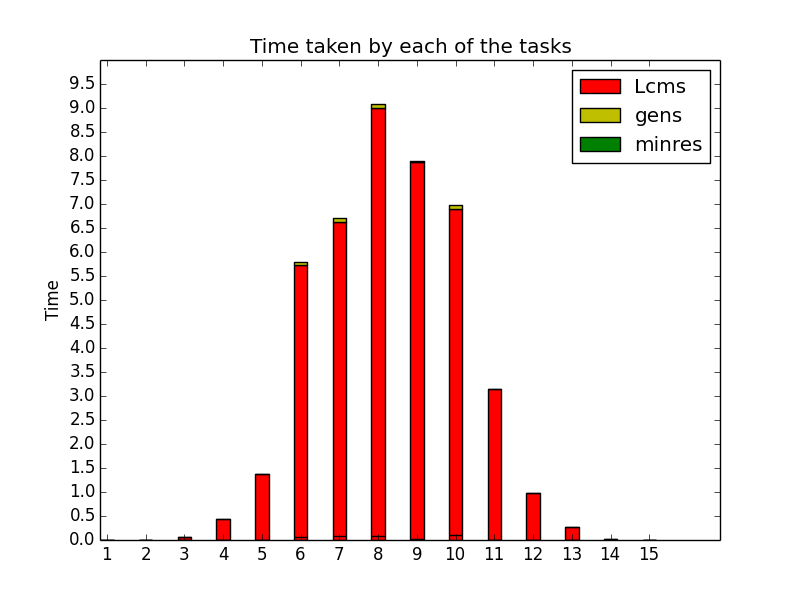}
\caption{Costs of each of the tasks for the cut ideal of $K_5$} %the complete graph $K_5$}
\label{fig:computations2}
\end{center}
\end{figure}

\section{Failure distributions and signatures}\label{sec:statistics}

%%%%%%%%%%%%%%%%%%%%%%%%%%%%%%%%%%%%%%%%%%%%%%%%%%
We finish our paper by applying the above considerations on the $\lcm$-lattice of system ideals to probability analysis. We have a tool for computing moments of the probability distribution of the number of elementary cuts (failures) of a given system.

\subsection{Computing means and higher moments}
In the case of binary (squarefree) ideals, which includes the case of the $k$-out-of-$r$ ideal and the cut ideal of a network in this paper,
the computation of the moment of the distribution of $Y$, the number of elementary cuts, is straightforward.

\begin{Lemma}\label{lem:mean}
Let ${\mathcal C}$ be the set of elementary cuts for a network reliability
problem, and let $Y$ be the number of elementary cuts. Then under the Erd\"os-R\'enyi independence model with probability $p$, the expectation of $Y$  considered as a random variable, is given by
$$\mbox{E}(Y) = \sum_{\alpha \in \mathcal C} p^{|\alpha|},$$
where $ |\alpha|$ is the degree of $\alpha$.
\end{Lemma}
\begin{proof}
Let $X_{\alpha}$ be the indicator function for the cut ${\alpha}$. Then $Y = \sum_{\alpha \in \mathcal C} X_{\alpha}$ and  $\mbox{E}(Y) = \mbox{E} (\sum_{\alpha \in \mathcal C} X_{\alpha}) = \sum_{\alpha \in \mathcal C} \mbox{E}(X_{\alpha})$. But for any $\alpha$, $X_{\alpha} = \prod_{i=1}^n X_i^{\alpha_i},$ where $X_i$ is the indicator function of the failure of the $i^{\rm th}$ component for $i=1, \ldots, n$. However, for any indicator function
$\mbox{prob} \{X_{\alpha} = 1\} = \mbox{E}( X_{\alpha})$ by independence in the Erd\"os-R\'enyi model.
Thus $\mbox{E}( X_{\alpha}) = \mbox{E}(\prod_{i=1}^n X_i^{\alpha_i}) = \prod_{i=1}^n \mbox{E}(X_i^{\alpha_i}) = p^{|\alpha|} $ which completes the proof. 
\end{proof}

\begin{Theorem}
For the complete graph $K_n$ the mean value is:
\begin{equation*}\label{orderpull} 
\mu_n=
\begin{cases}
\sum_{k=1}^{[\frac{n}{2}]}{n\choose k} p^{k(n-k)}\quad\quad\quad\quad\quad\quad\ {\rm \ \ if\  }n\  {\rm is\ odd};
\\
\\
\sum_{k=1}^{[\frac{n}{2}]-1}{n\choose k} p^{k(n-k)}+
\frac{1}{2}{n\choose \frac{n}{2}} p^{(\frac{n}{2})^2}\  {\rm \ \ if\ }n\  {\rm is\ even}.
\end{cases}
\end{equation*}
\end{Theorem}
\begin{proof}
By Lemma~\ref{lem:mean} we just need to count the number of elementary cuts, i.e. $2$-patitions of the graph. By 
\cite[Corollary 6.8]{Postnikov} this number is equal to $S(n,2)$ where %$S(n,k)$ 
denotes the Stirling number of the second kind (i.e. the number of ways to partition a set of $n$ elements into $k$ nonempty subsets). 
Note that for a partition with $k$ vertices in one part, and $n-k$ vertices in the other part, we have $k(n-k)$ edges between these two parts, and the number of such partitions is ${n\choose k}$. In case that $n$ is even and $k=\frac{n}{2}$, we have to divide this number by two because of double counting.
\end{proof}

\begin{Theorem}
For the sequential  $k$-out-of-$n$ ideal $I_{k,n}$, the mean value is $\mu_{k,n}=(n-k+1) p^k$.
\end{Theorem}
\begin{proof}
In this case we have $n-k+1$ cut momomials of degree $k$. 
Thus the result follows by Lemma~\ref{lem:mean}.
\end{proof}
%For the complet graph every elementary cut arises from a bipartite graph corrresponding to a two-group partition. The mean is:
%Fatemeh photographed the answer!

\begin{Remark}{\rm 
The same method can be used to obtain higher order moments. For the second non-central moment, $\mu_2=\mbox{E}(Y^2)$ we write
\begin{eqnarray*}
\mu_2 & = & \mbox{E}\left(\left(\sum_{\alpha \in \mathcal C}X_{\alpha}\right)^2\right) \\
& = &\mbox{E} \left( \sum_{\alpha \in \mathcal C} X_{\alpha}^2 + \sum_{\alpha, \beta \in {\mathcal C}\atop \alpha \neq \beta} X_{\alpha} X_{\beta}\right).
\end{eqnarray*}
Noting that $X_{\alpha}^2= X_{\alpha}$ and $X_{\alpha} X_{\beta} = X_{\alpha \wedge \beta}$, and using the argument above we have
\[
\mu_2 = \sum_{\alpha \in \mathcal C} p^{|\alpha|} + 2 \sum_{\gamma \in {\mathcal C}_2} p^{|\gamma|},
\]
where ${\mathcal C}_2$ is the set of raw monomial from the first $\lcm$ list.
\smallskip

Similarly we see that the $k$-th non-central moment can be written in terms of the  $\lcm$-ideals up to the $k^{\rm th}$ level. 
%We omit the full combinatorial formulae.
}
\end{Remark}

\subsection{Examples}
We develop the case of sequential $2$-out-of-$6$ ideal with
all the accompanying $\lcm$-ideals in more details. We compute the Hilbert series $H_k(s,t)$ based on Aramova-Herzog-Hibi
version of the Hilbert series for the  $\lcm$-ideals $I_k$ for $k=1, \ldots, 5$. The ideal is
$I=\langle x_1x_2,x_2x_3,x_3x_4,x_4x_5,x_5x_6 \rangle.$
The Hilbert series are
\begin{equation*}\begin{split}
H_1(s,t)& =1-( x_1x_2+x_2x_3 + x_3x_4 + x_4x_5 + x_5x_6)t \\
&+ ( x_1x_2x_3+ x_1x_2x_3x_4+ x_1x_2x_4x_5+ x_1x_2x_5x_6
+x_2x_3x_4+ x_2x_3x_4x_5\\& +x_2x_3x_5x_6+ x_3x_4x_5 + x_3x_4x_5x_6+ x_4x_5x_6+ x_1x_2x_3x_4)t^2 \\
& +(x_1x_2x_3x_4x_5+x_1x_2x_3x_5x_6
+x_1x_2x_3x_4x_5 +x_1x_2x_3x_4x_5x_6\\& +x_1x_2x_4x_5x_6 +x_2x_3x_4x_5
+x_2x_3x_4x_5x_6+x_2x_3x_4x_5x_6 + x_3x_4x_5x_6)t^3 \\
&+( x_1x_2x_3x_4x_5+ x_1x_2x_3x_4x_5x_6 + x_1x_2x_3x_4x_5x_6 \\
& + x_1x_2x_3x_4x_5x_6 + x_2x_3x_4x_5x_6)t^4 -x_1x_2x_3x_4x_5x_6t^5 \\
&\\
H_2(s,t)& = 1- (x_1x_2x_3+x_1x_2x_3x_4+x_1x_2x_4x_5+ x_1x_2x_5x_6 + x_2x_3x_4
+ x_2x_3x_4x_5 \\& +x_2x_3x_5x_6 +x_3x_4x_5 +x_3x_4x_5x_6+ x_4x_5x_6)t+
(2x_1x_2x_3x_4+ 2x_1x_2x_3x_4x_5\\
& + 2x_1x_2x_3x_5x_6 + 2x_1x_2x_3x_4x_5 + 2x_1x_2x_3x_4x_5x_6
+ 2x_1x_2x_4x_5x_6\\
& + 2x_2x_3x_4x_5+ 2x_2x_3x_4x_5x_6+ 2x_2x_3x_4x_5x_6+ 2x_3x_4x_5x_6)t^2 \\
& - ( 3x_1x_2x_3x_4x_5+3x_1x_2x_3x_4x_5x_6 +3x_1x_2x_3x_4x_5x_6 +3x_1x_2x_3x_4x_5x_6\\
& +3x_2x_3x_4x_5x_6)t^3 + 4x_1x_2x_3x_4x_5x_6 t^4 \\
\end{split}\end{equation*}
\begin{equation*}\begin{split}
H_3(s,t)& = 1 -( x_1 x_2x_3x_4+x_1x_2x_3x_4x_5 + x_1x_2x_3x_5x_6 + x_1x_2x_3x_4x_5 \\
& + x_1x_2x_3x_4x_5x_6 +x_1x_2x_4x_5x_6+x_2x_3x_4x_5 +x_2x_3x_4x_5x_6 + x_2x_3x_4x_5x_6\\
& +x_3x_4x_5x_6)t + (3x_1x_2x_3x_4x_5 + 3x_1x_2x_3x_4x_5x_6+ 3x_1x_2x_3x_4x_5x_6 \\
& + 3x_1x_2x_3x_4x_5x_6 + 3x_2x_3x_4x_5x_6)t^2 -6x_1x_2x_3x_4x_5x_6 t^3 \\
& \\
%\end{split}\end{equation*}
%\begin{equation*}\begin{split}
H_4(s,t) & = 1- (x_1x_2x_3x_4x_5+x_1x_2x_3x_4x_5x_6 +x_1x_2x_3x_4x_5x_6 + x_1x_2x_3x_4x_5x_6 \\
& +x_2x_3x_4x_5x_6)t +4x_1x_2x_3x_4x_5x_6t^2 \\
& \\
H_5 (s,t)& = 1 -x_1x_2x_3x_4x_5x_6. \\
\end{split}\end{equation*}
Under the usual failure model, of independent component (edge) failure with probability $p$, if $Y$ is the number of failures the survivor are given by:
\begin{equation*}\begin{split}
P_k & = F(k) \\
& = \mathbb{P}\{Y \geq k\}\\
& = 1 - H_k(p,t).
\end{split}\end{equation*}
These are given by:
\begin{equation*}\begin{split}
P_1 &= 5p^2-4p^3-3p^4+4p^5-p^6 \\
P_2 & = 4p^3-6p^5+3p^6 \\
P_3 & = 3p^4 - 2p^6 \\
P_4 & = 2p^5 -p^6 \\
P_5 & = p^6
\end{split}\end{equation*}

\begin{figure}[htbp]
\begin{center}
\includegraphics[scale=0.3]{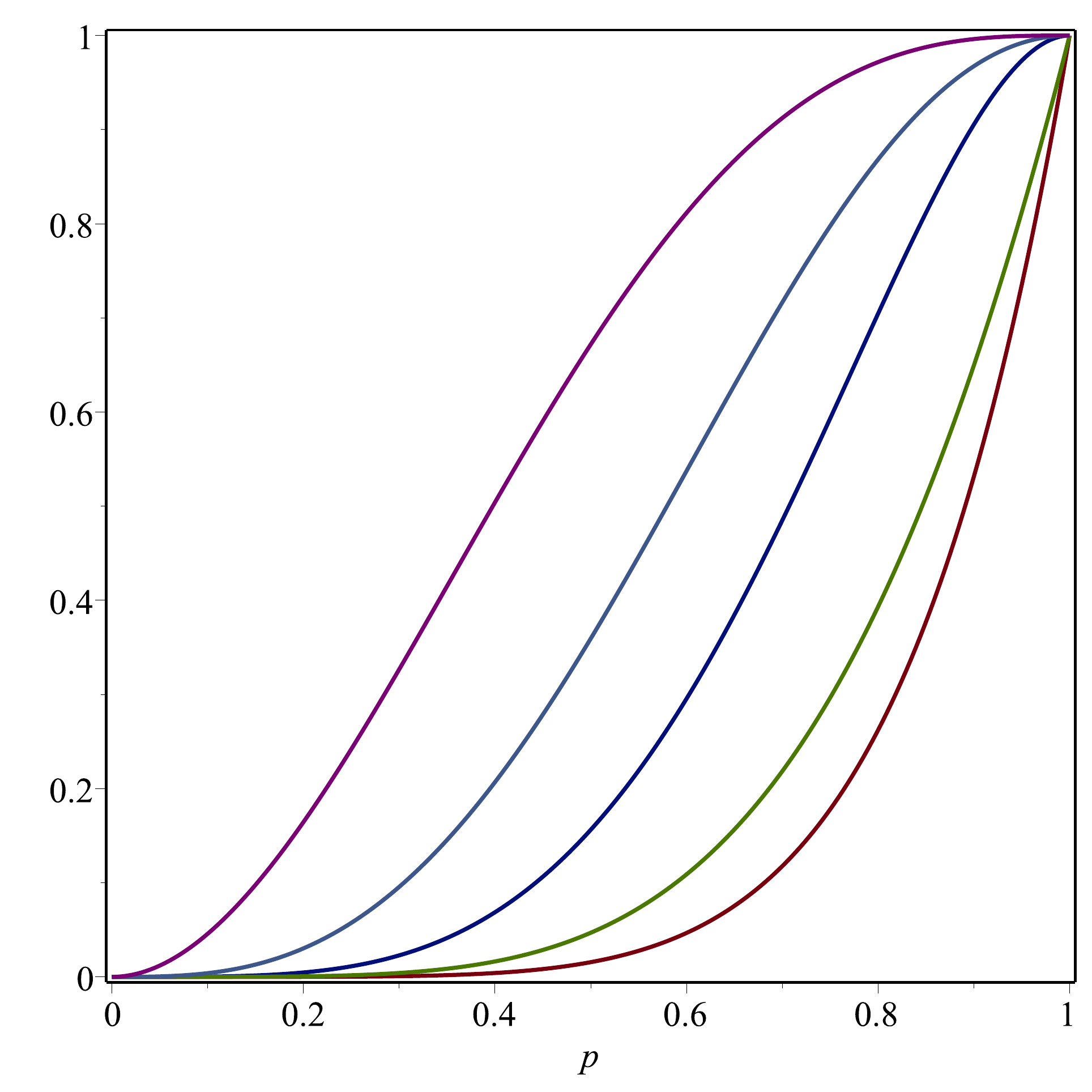}
\caption{Plots of $P_k(p)$, $k=1,\ldots,5$, from left to right, respectively}
\label{sizesConsecutive}
\end{center}
\end{figure}

Setting $p_0 = 1- P_1$, $p_k = P_k-P_{k+1},\;k = 1,..., 4$ and $p_5 = P_5 $ we find the
(discrete) distribution as
\begin{equation*}\begin{split}
p_0 &=  1-5p^2+4p^3+3p^4-5p^5+p^6 \\
p_1 &=  5p^2-8p^3-3p^4+10p^5-4p^6 \\
p_2 & = 4p^3-3p^4-6p^5+5p^6 \\
p_3 & = 3p^4 -2p^5 -p^6\\
p_4 & = 2p^5 -2p^6 \\
p_5 & = p^6
\end{split}\end{equation*}

The mean of this distribution is $\mu_{6,2} = \sum_{i=0}^5 i p_i = 5 p^2$. The form of the distribution for general sequential $k$-out-of-$n$
was found relatively recently together with the general form of the mean, confirmed in our case: 
\[\mu_{n,k} = (n-k+1)p^k,\]
see \cite{Ling}. In probability theory and related areas this might be called ``the expected number of runs of size $k$ in a Bernoulli sequence of length $n$ and probability $p$". Care has to be taken, in accessing the literature, concerning the definition of a run. For example whether overlaps are counted, as here, this distinguishes from isolated runs of length $k$.

For the signature we need to find the  intersection
$\tilde{I}_k$ of the $k$-out-of-$n$ ideals for $k=1, \ldots,6$ with the sequential $k$-out-of-$n$ ideal $I_1$,
 noting that the latter is the basic failure ideal. A simple way to carry out this
calculation is to identify for each $k=1, \ldots, 6$ which monomial of degree $k$ occur in $I_1$.

For $k=1$, none of $x_1, \ldots, x_6$ lie in $I_1$ and for $k=2$ we obtain $\tilde{I}_2 = I_1$. 
Similarly:
\begin{equation*}\begin{split}
\tilde{I}_3 & = \langle x_1x_2 x_3, x_2x_3x_4, x_3x_4x_5,x_4x_5x_6 \rangle \\
\tilde{I}_4 & = \langle x_1x_2x_3x_4, x_1x_2x_4x_5, x_1x_2x_5x_6,x_2x_3x_4x_5,x_2x_3x_5x_6,x_3x_4x_5x_6 \rangle \\
\tilde{I}_5 & = \langle x_1x_2x_3x_4x_5, x_1x_2x_3x_5x_6,x_1x_2x_4x_5x_6, x_2x_3x_4x_5x_6 \rangle \\
\tilde{I}_6 & = \langle x_1x_2x_3x_4x_5x_6 \rangle.
\end{split}\end{equation*}
The associated cumulative probabilities including $\tilde{I}_2$ are:
\begin{equation*}\begin{split}
Q_2 & = 5p^2-4p^2-3p^4+4p^5-p^6 \\
Q_3 & = 4p^3 -3p^4 \\
Q_4 & = 6p^4 - 6p^5 + p^6 \\
Q_5 & = 4p^5 - 3p^6 \\
Q_6 & = p^6.
\end{split}\end{equation*}
Again by taking differences we obtain the raw signature probabilities as
\begin{equation*}\begin{split}
q_2 & = 5p^2-8p^3+4p^3-p^6\\
q_3 &  = 4p^3 -9p^4+6p^5-p^6\\
q_4 & = 6p^4-10p^5+4p^6\\
q_5 & = 4p^5-4p^6 \\
q_6 & = p^6.
\end{split}\end{equation*}
From these we compute the signatures as $s_j = \frac{q_j}{P_1}$ for $j=2,\ldots,6$.

One purpose of this paper is to present the $\{p_k\}_{k=1}^n$, namely the distribution of the number of
elementary cuts as an alternative ``signature" to the classical signature distribution, of the number of failed components, in the event of failure.
But we can also study systems by looking at several different types of signature, what might be called multivariate signature analysis. To make this point clear we compute, for the current example, the {\em joint} distribution, that is to say the distribution
of the bivariate random variables $(Y, Z)$, where $Y$ is the number of elementary cuts and $Z$ is the number of failed components, conditional on failure.

Thus  take $p=\frac{1}{2}$, which correspond to simple counting, because then every binary state vector has probability $\frac{1}{2^6}$. The table below
counts the multiplicity of each pair $(Y,Z)=(y,z)$,  which are also failure cells, 20 cells is all (blank cells denote zero). We note the close association between the two types of signature.

\vspace{3mm}
\begin{center}
\begin{tabular}{|c|c|c|c|c|c|c|}
  \hline
  % after \\: \hline or \cline{col1-col2} \cline{col3-col4} ...
  \tiny{6} &   &   &   &   & 1   \\ \hline
  \tiny{5} &   &   & 2 & 2 &     \\ \hline
  \tiny{4} &   & 3 & 3 &   &    \\ \hline
  \tiny{3} &   & 4 &   &   &    \\ \hline
  \tiny{2}  & 5 &   &   &   &    \\ \hline
 $^z\diagup_y$ & \tiny{1} & \tiny{2} & \tiny{3} & \tiny{4} & \tiny{5}   \\
\hline
\end{tabular}

\vspace{3mm}
Table 1. Elementary cuts via component failure, for the sequential $2$-out-of-$6$ system.
\end{center}

\section{Acknowledgements}
The second and third authors were partially supported by Ministerio de Economia y Competitividad, Spain, under grant MTM2013-41775-P.

\end{document}